\documentclass[11pt]{amsart}
\usepackage{latexsym, amsmath, amssymb,amsthm,amsopn,amsfonts}
\usepackage{version}
\usepackage{epsfig,graphics,color,graphicx,graphpap}
\usepackage{amssymb}
\usepackage{todonotes}
\usepackage{listings}
\usepackage{mathrsfs}
\usepackage[citecolor=blue]{hyperref}

\usepackage[framemethod=tikz]{mdframed}
\newcommand{\redsout}{\bgroup\markoverwith{\textcolor{red}{\rule[0.5ex]{2pt}{.4pt}}}\ULon}

\newcommand{\LC}{\left(}
\newcommand{\RC}{\right)}

\newcommand{\p}{\partial}

\numberwithin{equation}{section}
\newtheorem{theorem}{Theorem}[section]
\newtheorem{corollary}[theorem]{Corollary}
\newtheorem{proposition}[theorem]{Proposition}
\newtheorem{lemma}[theorem]{Lemma}

\newtheorem{remark}{Remark}[section]
 
\newcommand{\R}{\mathbb R}

\author[Lai]{Ru-Yu Lai}
\address{School of Mathematics, University of Minnesota, Minneapolis, MN 55455, USA}
\curraddr{}
\email{rylai@umn.edu}

\author[Uhlmann]{Gunther Uhlmann}
\address{Department of Mathematics, University of Washington, Seattle, WA 98195, USA}\address{HKUST Jockey Club Institute for Advanced Study, HKUST, Clear Water Bay, Kowloon, Hong Kong.}
\curraddr{ }
\email{gunther@math.washington.edu}

\author[Yang]{Yang Yang}
\address{Department of Computational Mathematics Science and Engineering, Michigan State University, East Lansing, MI 48824, USA}
\curraddr{}
\email{yangy5@msu.edu}
 
\title{Reconstruction of the collision kernel in the nonlinear Boltzmann equation}
 
\begin{document}

\maketitle
\begin{abstract}
We consider an inverse problem for the Boltzmann equation with nonlinear collision operator in dimensions $n\geq 2$. We show that the kinetic collision kernel can be uniquely determined from the incoming-to-outgoing mappings on the boundary of the domain provided that the kernel satisfies a monotonicity condition. Furthermore, a reconstruction formula is also derived.
The key methodology is based on the higher-order linearization scheme to reduce a nonlinear equation into simpler linear equations by introducing multiple small parameters into the original equation. 
 
\end{abstract}

\section{Introduction}
\subsection{Motivation.}
Kinetic theory describes the dynamics of a large number of particles from a microscopic point of view. 
In particular, kinetic theory enjoys a lot of unique properties and demonstrates complicated mathematical features, which put it at a very important place for scientific studies. 
Applications of the kinetic theory include the dynamics of dilute charged particles, the semiconductor device, and space plasma physics \cite{	CGR2011, Howes}.

Kinetic equations model the evolution of a many-body particle system by means of a single-particle distribution function. The collision operators are particularly crucial for approximating the underlying details of the many-body particle interactions. Among all collision operators, arguably the most well-known one is the Boltzmann collision operator that describes the binary particle interaction by a kinetic distribution $F=F(x,v)$ and takes the form
\begin{align*}
Q(F,F)= \int_{\R^3}\int_{\mathbb{S}^2} q(\xi,\theta) [F(x,u')F (x,v') - F (x,u)F (x,v)] \,d\omega du,
\end{align*}
where $\xi=|v-u|$ and $\cos\theta=(v-u)\cdot \omega /|v-u|$, $\omega\in\mathbb{S}^{2}$. 
The vectors 
\begin{align}\label{velocities after collision}
u'=u-[(u-v)\cdot\omega]\omega\ \ \hbox{and }\ v'=v+[(u-v)\cdot\omega]\omega
\end{align} 
denote velocities after a collision of particles having velocities $v,u$ before the collision.  
The function $q(\xi,\theta)$ is called the collision cross section (kernel) and its form depends on the species of particles. For example, in the hard potential, the collision kernel is
$
q(\xi,\theta)= |v-u|^\gamma q_0\LC{v-u\over|v-u|}\cdot\omega\RC,
$
where $0<\gamma\leq 1$ and
$\int_{\mathbb{S}^2}q_0(\theta\cdot \omega)d\omega<\infty$ for $\theta\in\mathbb{S}^2$.

In the forward problem, there have been substantial contributions in the mathematical study of various aspects of Boltzmann equations. These involve the existence and uniqueness of the solutions, the decay of solutions toward a Maxwellian, as well as the connection between the kinetic theory and fluid dynamics,  see for instance \cite{DV05, DLion89, GSR1, GSR2, Guo2010, IS84, Kim2011, MW99, Villanibook} and the references therein.

\subsection{Inverse Problem.}
The inverse problem for kinetic equations is to find out hidden properties of the unknown parameters in the equations from the experimental data. 
Due to the importance of the collision operator in the kinetic theory, there is an increasing interest in solving these problems. The aim here is to study the identification of the unknown collision kernel from indirect measurements on the boundary.

Let us describe the Boltzmann equation studied in this article. Let $\Omega\subset\R^n$ be a bounded domain with $C^\infty$ boundary $\p\Omega$ with $n\geq 2$. 
We consider the following Boltzmann equation:
\begin{align}\label{intro: boltzmann equ}
\left\{\begin{array}{ll}
v\cdot \nabla_x F = Q(F,F) & \hbox{in }\Omega\times\R^n,  \\
F=g & \hbox{on }\Gamma_-,
\end{array}\right.
\end{align}
where $F(x,v)$ is the distribution function that depends on the position $x\in \Omega$ and the velocity $v\in \R^n$. Throughout this paper, the collision operator takes the form
\begin{align}\label{def:collision}
    Q(H_1,H_2 )= \int_{\R^n}\int_{\mathbb{S}^{n-1}} B(v,u,\omega) [H_1(x,u')H_2(x,v') - H_1(x,u)H_2(x,v)] \,d\omega du,
\end{align}
where $B(v,u,\omega)$ is the collision kernel and $u',\,v'$ are defined in \eqref{velocities after collision}.

We denote the boundary operator $\mathcal{A}$ that maps from the incoming data $F\in C(\Gamma_-)$ on $\Gamma_-$ to the outgoing one on $\Gamma_+$ by 
\begin{align}\label{bdry operator}
\mathcal{A}: F|_{\Gamma_-} \mapsto F|_{\Gamma_+}\in C(\Gamma_+).
\end{align}
Here the sets $\Gamma_{\pm}$ are defined through
\begin{align*} 
\Gamma_{\pm}:=\{(x,v)\in \p \Omega\times \R^n: \pm \ n(x)\cdot v > 0\}, 
\end{align*}
where $n(x)$ is the unit outer normal to $\p\Omega$ at the point $x\in \p\Omega$. It follows from Section~\ref{sec:Preliminaries} that the boundary value problem \eqref{intro: boltzmann equ} is well-posed for small boundary data $g\in  C(\Gamma_-)$. Hence, the map $\mathcal{A}$ is well-defined within the class of small boundary data.
The inverse problem in this paper concerns the extraction of the information of the collision kernel $B$ from the incoming-to-outgoing boundary map $\mathcal{A}$.

There have been related investigations in inverse problems for kinetic equations. One widely studied one is the radiative transfer equation, a linear Boltzmann equation with the linear collision operator. Let us introduce the problem for the RTE briefly below. 
The main objective is to determine the optical parameters from the albedo operator, that is known as the associated boundary operator to the RTE.
In particular, the uniqueness and stability issues have been extensively addressed in the literature.
In \cite{IKun, CS1, CS2, CS3, CS98, SU2d}, the parameters are uniquely determined from the boundary measurements.
The key ingredient of such reconstruction mainly replies on the singular decomposition of the collision kernel that was developed in \cite{CS2,CS98}.
In terms of the stability, Lipschitz stability estimates were studied in \cite{Bal14, Bal10,  Bal18, BalMonard_time_harmonic, LLU2018, Machida14, Wang1999, ZhaoZ18}. Furthermore, this inverse problem for the RTE has also been investigated in the Riemannian setting, see for example \cite{AY15, DPSU07, MST10, MST10stability, MST11, McDowall04}.

To study inverse problems for nonlinear equations, there is a classical method introduced by Isakov \cite{Isakov93}. 
This method is to show that the boundary map for the nonlinear equation determines the analogue for its linearized equation. Then one can apply the existing result of inverse problems for such linearized equation to identify the unknown property.
However this method does not work for the inverse problem considered here.
Since the collision operator highly depends on the velocities before and after the collisions, the first linearization of \eqref{intro: boltzmann equ} is fundamentally different from the RTE (a linear Boltzmann equation). As a result, the previously known theory for the RTE does not provide direct help to determine the kernel of \eqref{def:collision}.

In \eqref{intro: boltzmann equ}, the nonlinear interaction in the collision operator \eqref{def:collision} introduces certain degree of difficulty to the investigation of the inverse problem. 
To overpass this difficulty, we introduce the higher-order linearization technique to the nonlinear Boltzmann equation. This technique employs nonlinearity as a tool in solving inverse problems for nonlinear equations. 
Its central idea is based on bringing in several small parameters into the data, and then differentiating the nonlinear equation with respect to these parameters to earn simpler linearized equations. 
In particular, the work \cite{KLU2018} discovered that the nonlinearity can be beneficial in solving the inverse problem for the nonlinear hyperbolic equation, see also \cite{CLOP,LUW2018} and the references therein. For the nonlinear elliptic equation, the works \cite{CNV19, Kang2002, Sun96, SU97} have studied the second order linearization of the nonlinear boundary map. Moreover, this method has been applied to study various inverse problems for elliptic equations with power-type nonlinearities in \cite{AZ17, FL2019, KU201909, KU2019, LLLS201903, LLLS201905}.

When one applies the higher-order linearization to the problem \eqref{intro: boltzmann equ}, one can expect that the analysis of recovering the kernel will be very different from the case for the elliptic equations with nonlinearity, such as $\Delta u + q(x)u^k =0$ for a positive integer $k$ in \cite{KU2019,LLLS201903}. The difference not only comes from the type of equations, but also the form of nonlinearity. 
Compared to the nonlinearity $u^k(x)$ in the elliptic equations, the nonlinearities $F(x,u')F(x,v')$ and $F(x,u)F(x,v)$ here depend on different variables. Thus, the second linearized equation has more terms than pure power-type nonlinearity.
Nevertheless, the unknown kernel only appears in the second linearization of \eqref{intro: boltzmann equ} and leaves the first linearization of \eqref{intro: boltzmann equ} plenty of freedom to choose its solutions. These turn out to be a crucial ingredient to determine the kernel. The detailed discussion is in Section~\ref{sec:recovery} and Section~\ref{sec:formula}.

\subsection{Main Results.}
Let $\Omega\subset\R^n$ be a bounded domain with $C^\infty$ boundary $\p\Omega$ with $n\geq 2$.
For $(x,v)\in \overline\Omega\times(\R^n\setminus\{0\})$, we define
$\tau_\pm(x,v)$ as the exit time from the point $x$ to the boundary $\p\Omega$ in the direction $\pm v$, namely,
$$
   \tau_\pm(x,v) :=\sup\{s\geq 0:\ x\pm sv \in \Omega  \}.
$$

Suppose that the kernel $B\in C(\R^n\times\R^n\times\mathbb{S}^{n-1})$ satisfies the following condition:
there exists a constant $M>0$ such that
\begin{align}\label{constraint B}
\tau_\pm(x,v) \int_{\R^n}\int_{\mathbb{S}^{n-1}}|B(v,u,\omega)| \,d\omega du < M<\infty
\end{align}
for all $(x,v)\in\overline\Omega\times\R^n$.
We now state the main results.

\begin{theorem}[Monotonicity uniqueness]\label{main thm}
Let $\Omega\subset\R^n$ be a bounded domain with $C^\infty$ boundary $\p\Omega$ with $n\geq 2$. 
Let the collision kernel $B_j\equiv B_j(v,u,\omega)$ be in $C(\R^n\times\R^n\times\mathbb{S}^{n-1})$ and satisfy \eqref{constraint B}. Let $\mathcal{A}_j$ be the boundary operator of the problem \eqref{intro: boltzmann equ} with the kernel $B$ replaced by $B_j$ for $j=1,2$.
Suppose that 
$$
\mathcal{A}_1(g)=\mathcal{A}_2(g) 
$$ 
for all $g\in C(\Gamma_-)$ with $\|g\|_{C(\Gamma_-)}<\varepsilon$, where $\varepsilon$ is a sufficiently small number.
If $B_1\geq B_2$ pointwisely in $\R^n\times\R^n\times\mathbb{S}^{n-1}$ (Monotonicity condition), then   
$$
B_1=B_2 \ \ \hbox{everywhere  in }\R^n\times\R^n\times\mathbb{S}^{n-1}.
$$	
\end{theorem}

We also have the following reconstruction formula for $B$ provided that $B$ satisfies some weak assumptions.
\begin{theorem}[Reconstruction formula]\label{thm:ReconstructionFormula}
Let $\Omega\subset\R^n$ be a bounded domain with $C^\infty$ boundary $\p\Omega$ with $n\geq 2$ and let $B\in C(\R^n\times\R^n\times\mathbb{S}^{n-1})$ satisfy \eqref{constraint B}. Suppose that $B$ is symmetric in incoming velocities and is also an even function of $\omega$, that is,
\begin{align}\label{assumption B}
   B(v,u,\omega) = B(u,v,\omega)\ \ \hbox{and}\ \ B(v,u,-\omega)=B(v,u,\omega).
\end{align} 
Then for any $(a,b,\theta)\in \mathcal{D}$ (defined in \eqref{def D}) in $\mathbb{R}^n\times\mathbb{R}^n\times\mathbb{S}^{n-1}$, we have  
	\begin{align} \label{thm_eq:SInTemrsOfB}
	&S(a,a - [(a-b)\cdot\theta]\theta, b + [(a-b)\cdot\theta]\theta)\notag \\
	&=|(a-b)\cdot\theta|^{-2} (B(a, b, \theta)+ (|a-b|^2-|(a-b)\cdot\theta|^2)^{-1} B(a, b, \widehat{P_{\theta^\perp}(a-b)}),
	\end{align}
	where we denote $\hat{z}:={z\over |z|}\in\mathbb{S}^{n-1}$ and $P_{\theta^\perp}(a-b) := (a-b)-[(a-b)\cdot\theta]\theta$, and the function $S$ (defined in \eqref{eq:Sv}) is determined by the boundary measurement $\mathcal{A}$ only.
\end{theorem}
 
Following immediately from Theorem~\ref{thm:ReconstructionFormula}, we obtain the uniqueness result if $B$ satisfies \eqref{constraint B}-\eqref{assumption B}.
\begin{corollary}[Uniqueness: two special cases]\label{Coro}
	Suppose that two collision kernels $B_1$ and $B_2$ satisfy \eqref{constraint B} and \eqref{assumption B} and have identical boundary measurements.
	Then $B_1=B_2$ in the following two cases: 
	\begin{enumerate}
		\item  the collision kernel $B=B(v,u)$ is independent of $\omega$;
		\item  the monotonicity condition is valid, such as $B_1\geq B_2$.
	\end{enumerate} 
\end{corollary}

\begin{remark}
Compared to Theorem~\ref{main thm}, the uniqueness result of Corollary~\ref{Coro} is constructive yet replies on an additional assumption \eqref{assumption B} since it is a direct consequence from the reconstruction formula stated in Theorem~\ref{thm:ReconstructionFormula}. 
\end{remark} 
 
We note that Theorem~\ref{main thm} illustrates the uniqueness result is valid if $B$ satisfies the monotonicity condition and its proof in Section~\ref{sec:recovery} replies on the suitable chosen Gaussian-like solutions to the first linearized equation of \eqref{intro: boltzmann equ}. To demonstrate Theorem~\ref{thm:ReconstructionFormula}, the methodology is based on a similar strategy in the study of the RTE by applying the solution having the boundary data that are only concentrated on the incoming direction. Thus, the information of the kernel $B$ can be carried out from the propagation of these particles.

\subsection{Outline.} The paper is organized as follows. Section~\ref{sec:Preliminaries} is devoted to prove fundamental results, including the well-posedness of \eqref{intro: boltzmann equ}. They will play an important role in the study of the determination of the kernel. In Section~\ref{sec:recovery}, we detail the analysis of the higher-order linearization scheme and provide the proof of Theorem~\ref{main thm}. Furthermore, the reconstruction formula in Theorem~\ref{thm:ReconstructionFormula} is presented and proved in Section~\ref{sec:formula} as well as the uniqueness results in two special cases are discussed under the same hypothesis.

\section{Preliminaries}\label{sec:Preliminaries}
In this section, we introduce the results that are essential for the investigation of the proposed inverse problem for \eqref{intro: boltzmann equ}. The main goal here is to establish the well-posedness for the boundary value problem \eqref{intro: boltzmann equ} with small incoming boundary data.

We first discuss the following lemma as preparation for the well-posedness result.
 
\begin{lemma}\label{Maximum principle}
Suppose that $\sigma\in L^\infty(\Omega)$ satisfies $\sigma\geq\sigma_0>0$ for positive constant $\sigma_0$. For $f\in C(\Omega\times\R^n)$ and $g\in C(\Gamma_-)$, the solution $F$ to 
	\begin{align}\label{ }
	\left\{ \begin{array}{ll} 
	v\cdot \nabla_x F +\sigma F= f& \hbox{in } \Omega\times\R^n,\\
	F=g &  \hbox{on }\Gamma_-,\\
	\end{array}\right. 
	\end{align}
has the form 
$$
   F(x,v) = e^{-\int^{\tau_-(x,v)}_0 \sigma(x-sv)ds}g(x-\tau_-(x,v)v, v)+\int^{\tau_-(x,v)}_0 e^{-\int^{s}_0 \sigma(x-\eta v)d\eta}  f(x-sv, v)\,ds
$$
and satisfies the estimate
\begin{align}\label{estimate F}
    \|F\|_{C(\Omega\times\R^n)} \leq  \|g\|_{C(\Gamma_- )} + C \|f\|_{C(\Omega\times\R^n)},
\end{align}
where $C$ depends only on $\sigma_0$.
\end{lemma}
\begin{proof}
It can be readily verified that $F(x,v)$ defined above is indeed a solution. Moreover, from the representation of $F$ and $\sigma\geq \sigma_0>0$, we have
\begin{align*} 
|F(x,v)| &= | e^{-\int^{\tau_-(x,v)}_0 \sigma(x-sv)ds}g(x-\tau_-(x,v)v, v)+\int^{\tau_-(x,v)}_0 e^{-\int^{s}_0 \sigma(x-\eta v)d\eta}  f(x-sv, v)\,ds |\\
&\leq \|g\|_{C(\Gamma_-)} + \|f\|_{C(\Omega\times\R^n)}\int^{\tau_-(x,v)}_0 e^{-\sigma_0 s}\,ds \\
&\leq \|g\|_{C(\Gamma_-)} + { 1\over \sigma_0}  \|f\|_{C(\Omega\times\R^n)} \LC 1- e^{-\sigma_0 \tau_-(x,v)}  \RC
\end{align*}
for any $(x,v)\in \Omega\times\R^n$. Thus, the estimate \eqref{estimate F} holds. 
\end{proof}

\begin{remark}\label{remark}
	We note that when $\sigma=0$, the solution takes the form $$F(x,v)=g(x-\tau_-(x,v)v, v)+\int^{\tau_-(x,v)}_0  f(x-sv, v)\,ds.$$ Then it is clear that 
	\begin{align}\label{estimate F 2}
	|F(x,v)| \leq  \|g\|_{C(\Gamma_- )} +  |\int^{\tau_-(x,v)}_0 f(x-sv,v)\,ds | \ \ \hbox{for all }x\in\Omega,\ v\in\R^n .
	\end{align}

\end{remark}

\subsection{Well-posedness} 
We consider the in-flow boundary condition for the Boltzmann equation
\begin{align}\label{Boltzmann equ}
	\left\{ \begin{array}{ll} 
    v\cdot \nabla_x F = Q(F,F)& \hbox{in } \Omega\times\R^n,\\
	F=g  &  \hbox{on }\Gamma_-,\\
	\end{array}\right. 
\end{align}
where the collision operator $Q$ is defined as in~\eqref{def:collision}.

We show the boundary value problem for \eqref{Boltzmann equ} is well-posed for small boundary data.
\begin{theorem}[Well-posedness of the Boltzmann equation]\label{thm:well posedness}
Let $\Omega\subset \R^n$, $n\geq2$ be a bounded domain with $C^\infty$ boundary $\p\Omega$.
Suppose that $B$ satisfies \eqref{constraint B}. 
Then there exists $\varepsilon>0$ such that when
\begin{align}\label{small boundary}
    g\in \mathcal{X}:=\{g\in C(\Gamma_-):\ \|g\|_{C(\Gamma_-)}\leq \varepsilon\},
\end{align}
the boundary value problem \eqref{Boltzmann equ} has a unique solution $F$. Moreover, there exists a constant $C>0$, independent of $g$, such that 
$$
    \|F\|_{C(\Omega\times\R^n)} \leq C \|g\|_{C(\Gamma_-) }.
$$
\end{theorem}
\begin{proof}
We utilize the contraction mapping principle to show the existence of solution to \eqref{Boltzmann equ}.
	
To this end, we first take any $g\in C(\Gamma_-)$ satisfying $\|g\|_{C(\Gamma_-)}\leq \varepsilon$ with $\varepsilon$ to be determined later, then there exists a unique solution $F_0$ to the equation
\begin{align}\label{Boltzmann equ 0}
\left\{ \begin{array}{ll} 
v\cdot \nabla_x F_0  = 0 & \hbox{in } \Omega\times\R^n,\\
F_0=g &  \hbox{on }\Gamma_-,\\
\end{array}\right. 
\end{align}
and $F_0$ satisfies
\begin{align}\label{estimate F0}
   \|F_0\|_{C(\Omega\times\R^n)}\leq  \|g\|_{C(\Gamma_-)}\leq \varepsilon.
\end{align}

Second, if $F$ is the solution to \eqref{Boltzmann equ}, then we have that $G:=F-F_0$ satisfies
\begin{align}\label{Boltzmann equ 1}
\left\{ \begin{array}{ll} 
v\cdot \nabla_x G = Q(F_0+G,F_0+G)=:\mathcal{F}(G) & \hbox{in } \Omega\times\R^n,\\
G=0 &  \hbox{on }\Gamma_-.\\
\end{array}\right. 
\end{align}
We denote by $\mathcal{L}^{-1}$ the solution operator to \eqref{Boltzmann equ 1} and, moreover, from \eqref{estimate F 2}, it satisfies
\begin{align}\label{estimate L}
|\mathcal{L}^{-1}(\mathcal{F}(G)) (x,v)|\leq C \tau_-(x,v)\|\mathcal{F}(G)(\cdot,v)\|_{C(\Omega)} 
\end{align}
for all $x\in\Omega$ and $v\in \R^n$.

Now we will show that $\mathcal{L}^{-1}\circ \mathcal{F}$ is a contraction map on a suitable subset of $C(\Omega\times\R^n)$.
We first define the subspace $\mathfrak{X}$ of $C(\Omega\times\R^n)$ by
$$
    \mathfrak{X}=\{\varphi\in C(\Omega\times\R^n):\ \varphi|_{\Gamma_-}=0,\ \|\varphi\|_{C(\Omega\times\R^n)}\leq \delta \}
$$
with some constant $\delta>0$ to be determined later. 
To simplify the notation, we further define an operator $\mathcal{M}$ on $\mathfrak{X}$ by
$$
    \mathcal{M} (\varphi) = (\mathcal{L}^{-1} \circ\mathcal{F})(\varphi)
$$
for any $\varphi\in \mathfrak{X}$.
From the direct computations, \eqref{estimate F0}, and \eqref{estimate L}, we obtain
\begin{align*}
&|\mathcal{M}(\varphi)(x,v)|\\
&= |\mathcal{L}^{-1} (Q(F_0+\varphi,F_0+\varphi))(x,v)|\\
&\leq C \tau_-(x,v) \|Q(F_0+\varphi,F_0+\varphi)(\cdot,v)\|_{C(\Omega)}\\
&= C\tau_-(x,v)\|\int_{\R^n}\int_{\mathbb{S}^{n-1}} B(v,u,w) [(F_0+\varphi)(x,u')(F_0+\varphi)(x,v') \\
&\hskip5cm - (F_0+\varphi)(x,u)(F_0+\varphi)(x,v)]\,dwdu \|_{C(\Omega)}\\
&\leq C \tau_-(x,v)\LC\int_{\R^n}\int_{\mathbb{S}^{n-1}}|B(v,u,w)|\, dwdu\RC (\varepsilon +\delta )^2\\
&\leq CM(\varepsilon+\delta)^2\ \ \hbox{for all }(x,v)\in \Omega\times \R^n,
\end{align*}
where the last inequality is due to \eqref{constraint B}.
Thus, we have
$$
\|\mathcal{M}(\varphi)\|_{C(\Omega\times\R^n)}  \leq CM(\varepsilon+\delta)^2.
$$
Moreover, for any $\varphi_1,\varphi_2\in \mathfrak{X}$, we also estimate  
\begin{align*}
&|\mathcal{M}(\varphi_1)(x,v) - M(\varphi_2) (x,v)| \\
&\leq \tau_-(x,v)|(Q(F_0+\varphi_1,F_0+\varphi_1)(x,v)-Q(F_0+\varphi_2,F_0+\varphi_2)(x,v)) | \\
&\leq C \tau_-(x,v)\int_{\R^n}\int_{\mathbb{S}^{n-1}}B |(v,u,w)|\, dwdu \Big(  4\|F\|_{C(\Omega\times\R^n)}\|\varphi_1-\varphi_2\|_{C(\Omega\times\R^n)} \\
&\quad + 2\|\varphi_1\|_{C(\Omega\times\R^n)}\|\varphi_1-\varphi_2\|_{C(\Omega\times\R^n)} + 2\|\varphi_2\|_{C(\Omega\times\R^n)} \|\varphi_1-\varphi_2\|_{C(\Omega\times\R^n)} \Big)\\
&\leq CM (4\varepsilon +4\delta )\|\varphi_1-\varphi_2\|_{C(\Omega\times\R^n)}\ \ \hbox{for all }(x,v)\in \Omega\times \R^n .
\end{align*}
If we choose $1>\varepsilon>0$ and $\delta>0$ such that $\delta<\varepsilon$,
$$
    C M  (\varepsilon+\delta)^2 \leq \delta,  
$$
and 
$$
 C M (4\varepsilon +4\delta )<1,
$$
then this leads to that $\mathcal{M}$ maps $\mathfrak{X}$ into itself and, moreover, 	is a contraction map on $\mathfrak{X}$.
By the contraction mapping principle, there exists a unique fixed point $\hat{F}\in \mathfrak{X}$ of $\mathcal{M}$ such that 
$\hat{F}$ is the solution of \eqref{Boltzmann equ 1}. In particular, from \eqref{estimate F0} and \eqref{estimate L}, one can derive that
\begin{align*} 
   \|\hat{F}\|_{C(\Omega\times\R^n)} &\leq CM ( \|g\|^2_{C(\Gamma_-)} + 2\|g\|_{C(\Gamma_-)}\|\hat{F}\|_{C(\Omega\times\R^n)} + \|\hat{F}\|^2_{C(\Omega\times\R^n)})\\
   &\leq CM\varepsilon \|g\|_{C(\Gamma_-)} +CM(2\varepsilon+\delta )\|\hat{F}\|_{C(\Omega\times\R^n)}.
\end{align*} 
We further require that $\varepsilon$ and $\delta$ satisfy $2\varepsilon+\delta \leq \gamma <1$ for some constant $\gamma$, we obtain
\begin{align*} 
\|\hat{F}\|_{C(\Omega\times\R^n)} \leq C\|g\|_{C(\Gamma_-)} .
\end{align*} 

Finally, we conclude that $F=F_0+\hat{F}$ is a solution of \eqref{Boltzmann equ} and satisfies the estimate
\begin{align*} 
    \|F\|_{C(\Omega\times \R^n)}\leq \|F_0\|_{C(\Omega\times \R^n)}+\|\hat{F}\|_{C(\Omega\times \R^n)}\leq C\|g\|_{C(\Gamma_-)} .
\end{align*}
This completes the proof.
\end{proof}

\section{Determination of the collision kernel}\label{sec:recovery} 
In this section, we will first perform the higher order linearization to the nonlinear Boltzmann equation. Under suitable constraints on the kernel, we will be able to uniquely determine the collision kernel from the boundary data.

\subsection{Linearization}
Since the nonlinearity in \eqref{intro: boltzmann equ} is quadratic-like, it is sufficiently to take parameters $\varepsilon=(\varepsilon_1,\varepsilon_2)$. For sufficiently small constants $\varepsilon_1,\,\varepsilon_2>0$ and $g_1,\,g_2\in C(\Gamma_-)$, by Theorem~\ref{thm:well posedness}, there exists a unique solution $F=F(x,v;\varepsilon)$ of the boundary value problem
\begin{align}\label{Boltzmann equ small data}
\left\{ \begin{array}{ll} 
v\cdot \nabla_x F = Q(F,F)& \hbox{in } \Omega\times\R^n,\\
F=\varepsilon_1 g_1+\varepsilon_2 g_2 &  \hbox{on }\Gamma_-,\\
\end{array}\right. 
\end{align}
and, specifically, the solution satisfies
$$
\|F\|_{C(\Omega\times\R^n)}\leq C \varepsilon_1\|g_1\|_{C(\Gamma_-)}+C\varepsilon_2\|g_2\|_{C(\Gamma_-)}.
$$
Next, let $V^{(k)}$ for $k=1,2$ be the solution of 
\begin{align}\label{Boltzmann equ V}
\left\{ \begin{array}{ll} 
v\cdot \nabla_x V^{(k)}= 0 & \hbox{in } \Omega\times\R^n,\\
V^{(k)}= g_k &  \hbox{on }\Gamma_-,\\
\end{array}\right. 
\end{align}
and then it satisfies
$$
\|V^{(k)}\|_{C(\Omega\times\R^n)}\leq C  \|g_k\|_{C(\Gamma_-)}.
$$
Lastly, we consider $W$ to be the solution to the boundary value problem
\begin{align}\label{Boltzmann equ linear W}
\left\{ \begin{array}{ll} 
v\cdot \nabla_x W  = S(x,v)  & \hbox{in } \Omega\times\R^n,\\
W= 0& \hbox{on }\Gamma_-,\\
\end{array}\right. 
\end{align}
where the function $S$ is denoted by
\begin{align}\label{eq:S}
S(x,v)&:=\int_{\R^n}\int_{\mathbb{S}^{n-1}} B(v,u,\omega) [V^{(1)}(x,v')V^{(2)}(x,u')+V^{(1)}(x,u')V^{(2)}(x,v') \nonumber \\
&\hskip4cm - V^{(1)}(x,u)V^{(2)}(x,v)-V^{(1)}(x,v)V^{(2)}(x,u)]\,d\omega du.  
\end{align}

In the following lemma, we show that the quotient $F/\varepsilon_k$ converges to $V^{(k)}$ in Lemma~\ref{lemma:1st appro} as well as we justify the approximation of the second derivatives of $F$ with respect to $\varepsilon$ in Lemma~\ref{lemma:2nd appro}.

Before starting the lemma, we denote the following functions:
\begin{align}\label{notation F}
F(x,v)=F(x,v;\varepsilon),\ \  
F^{(1)}(x,v)= F(x,v; \varepsilon_1,0), \ \
F^{(2)}(x,v)= F(x,v;0,\varepsilon_2). 
\end{align}
\begin{lemma}\label{lemma:1st appro}
Suppose that the assumptions in Theorem~\ref{thm:well posedness} hold, then we get
\begin{align}\label{asym 1st}
	\lim_{\varepsilon_1\rightarrow 0 }\|\varepsilon_1^{-1}F^{(1)} - V^{(1)}\|_{C(\Omega\times\R^n)} =0,
\end{align}
\begin{align}\label{asym 1st 0}	
	\lim_{\varepsilon_2\rightarrow 0 }\|\varepsilon_2^{-1}F^{(2)} - V^{(2)}\|_{C(\Omega\times\R^n)} =0.
\end{align}
Similarly, we also have
\begin{align}\label{asym 1st 2}
\lim_{\varepsilon \rightarrow 0 }\|\varepsilon_1^{-1}(F -F^{(2)} ) - V^{(1)}\|_{C(\Omega\times\R^n)} =0,
\end{align}
and 
\begin{align}\label{asym 2nd 2}
\lim_{\varepsilon \rightarrow 0 }\|\varepsilon_2^{-1}(F -F^{(1)} )- V^{(2)}\|_{C(\Omega\times\R^n)} =0.
\end{align}
\end{lemma} 
\begin{proof}
	We first consider the difference of \eqref{Boltzmann equ small data} and \eqref{Boltzmann equ V} for $k=1$, then we have
	\begin{align}\label{lemma: Boltzmann equ small data}
	\left\{ \begin{array}{ll} 
	v\cdot \nabla_x (\varepsilon_1^{-1} F-V^{(1)}) = \varepsilon_1^{-1}Q(F,F)& \hbox{in } \Omega\times\R^n,\\
	\varepsilon_1^{-1}F-V^{(1)}= \varepsilon_1^{-1}\varepsilon_2 g_2 &  \hbox{on }\Gamma_-.\\
	\end{array}\right. 
	\end{align}
	By Remark~\ref{remark} and \eqref{constraint B}, we have
	\begin{align*} 
	&\|\varepsilon_1^{-1} F-V^{(1)}\|_{C(\Omega\times\R^n)}\\
	&\leq \|\varepsilon_1^{-1}\varepsilon_2 g_2\|_{C(\Gamma_-)}+C \|\int_0^{\tau_-(x,v)}\varepsilon_1^{-1}Q(F,F)(x-sv,v)ds\|_{C(\Omega\times\R^n)}\\
	&\leq \|\varepsilon_1^{-1}\varepsilon_2 g_2\|_{C(\Gamma_-)}+C\varepsilon^{-1}_1\|F\|^2_{C(\Omega\times\R^n)}\max_{\Omega\times\R^n}\LC  \tau_-(x,v) \int_{\R^n}\int_{\mathbb{S}^{n-1}}|B| \, d\omega du\RC \\
	&\leq \|\varepsilon_1^{-1}\varepsilon_2 g_2\|_{C(\Gamma_-)}+C \varepsilon^{-1}_1(\varepsilon_1\|g_1\|_{C(\Gamma_-)}+ \varepsilon_2\|g_2\|_{C(\Gamma_-)})^2M .
	\end{align*}
	Let $\varepsilon_2\rightarrow 0$ and then we have
	$$
	\|\varepsilon_1^{-1} F(x,v;\varepsilon_1,0)-V^{(1)} \|_{C(\Omega\times\R^n)}\rightarrow 0\qquad \hbox{when }\varepsilon_1\rightarrow0.
	$$
	Similarly, for $k=2$, it also leads to  $$
	\|\varepsilon_2^{-1} F(x,v;0,\varepsilon_2)-V^{(2)} \|_{C(\Omega\times\R^n)}\rightarrow 0\qquad \hbox{when }\varepsilon_2\rightarrow0.
	$$

	Further, to show the second limits, we note that $\varepsilon_1^{-1} (F- F^{(2)})-V^{(1)}$ satisfies the problem 
		\begin{align}\label{lemma: Boltzmann equ small data 2}
	\left\{ \begin{array}{ll} 
	v\cdot \nabla_x (\varepsilon_1^{-1} (F- F^{(2)})-V^{(1)}) = \varepsilon_1^{-1}(Q(F,F) - Q(F^{(2)},F^{(2)}))& \hbox{in } \Omega\times\R^n,\\
	\varepsilon_1^{-1}(F-F^{(2)})-V^{(1)}= 0 &  \hbox{on }\Gamma_-.\\
	\end{array}\right. 
	\end{align}
	By a direct computation and applying Remark~\ref{remark} and \eqref{constraint B} again, we obtain the following estimate
	\begin{align*}
	    &\|\varepsilon_1^{-1} (F- F^{(2)})-V^{(1)}\|_{C(\Omega\times\R^n)}\\
	    &\leq CM (\|F\|_{C(\Omega\times\R^n)}+\|F^{(2)}\|_{C(\Omega\times\R^n)}) (\|\varepsilon_1^{-1} (F- F^{(2)})-V^{(1)}\|_{C(\Omega\times\R^n)} + \|V^{(1)}\|_{C(\Omega\times\R^n)} ).
	\end{align*}
	Thus, we get
	$$
	 (1- C\varepsilon_1-C\varepsilon_2)\|\varepsilon_1^{-1} (F- F^{(2)})-V^{(1)}\|_{C(\Omega\times\R^n)}\leq (C\varepsilon_1+C\varepsilon_2 ) \|V^{(1)}\|_{C(\Omega\times\R^n)},
	$$
	that goes to zero when $\varepsilon\rightarrow 0$. This completes the proof of \eqref{asym 1st 2}.

	Following a similar computation as above, we can obtain \eqref{asym 2nd 2}.
\end{proof}

\begin{lemma}\label{lemma:2nd appro}
Moreover, we obtain 
\begin{align}\label{asym 2nd}
\lim_{\varepsilon\rightarrow 0 }\|(\varepsilon_1 \varepsilon_2)^{-1}(F - F^{(2)} -F^{(1)} )- W\|_{C(\Omega\times\R^n)} =0.
\end{align}
\end{lemma}
\begin{proof}
We denote the function
$$
   \mathcal{G} :=(\varepsilon_1 \varepsilon_2)^{-1}(F(x,v;\varepsilon) - F^{(2)}(x,v)-F^{(1)}(x,v) )- W(x,v),
$$
then $\mathcal{G}$ satisfies
\begin{align}\label{lemma: Boltzmann G}
\left\{ \begin{array}{ll} 
    v\cdot \nabla_x \mathcal{G} = \mathcal{H} & \hbox{in } \Omega\times\R^n,\\
    \mathcal{G} = 0 &  \hbox{on }\Gamma_-,\\
\end{array}\right. 
\end{align}
where we used the notations defined in \eqref{notation F} and we define 
$$
    \mathcal{H}:= (\varepsilon_1\varepsilon_2)^{-1}(Q(F,F)-Q( F^{(2)}, F^{(2)})-Q( F^{(1)}, F^{(1)})) - Q(V^{(1)},V^{(2)})-Q(V^{(2)},V^{(1)}) .
$$
By using Remark~\ref{remark} and \eqref{constraint B} again, it leads to
\begin{align}\label{mathcal G}
 \|\mathcal{G}\|_{C(\Omega\times\R^n)}  
 \leq CM \| \mathcal{H}_1\|_{C(\Omega\times\R^n)}+CM \| \mathcal{H}_2\|_{C(\Omega\times\R^n)},
\end{align}
where 
\begin{align*}
\mathcal{H}_1: = \,&(\varepsilon_1\varepsilon_2)^{-1}(F(x,u')F(x,v') - F^{(2)}(x,u') F^{(2)}(x,v') -  F^{(1)}(x,u') F^{(1)}(x,v') )\\
&	-V^{(1)}(x,u')V^{(2)}(x,v')-V^{(1)}(x,v')V^{(2)}(x,u'),
\end{align*}
and 
\begin{align*}
\mathcal{H}_2: = \,&(\varepsilon_1\varepsilon_2)^{-1}( F(x,u)F(x,v) - F^{(2)}(x,u) F^{(2)}(x,v) -  F^{(1)}(x,u) F^{(1)}(x,v)) \\
& - V^{(1)}(x,u)V^{(2)}(x,v)-V^{(1)}(x,v)V^{(2)}(x,u).
\end{align*}
We observe that 
\begin{align}\label{2nd identity}
& \lim_{\varepsilon\rightarrow 0} (\varepsilon_1\varepsilon_2)^{-1}(F(x,u')F(x,v') - F^{(2)}(x,u') F^{(2)}(x,v') -  F^{(1)}(x,u') F^{(1)}(x,v') )\notag\\
&= \lim_{\varepsilon\rightarrow 0} \varepsilon_1^{-1}( F(x,u')- F^{(2)}(x,u')) \varepsilon^{-1}_2(F(x,v')- F^{(1)}(x,v') )-(\varepsilon_1\varepsilon_2)^{-1} F^{(2)}(x,u') F^{(1)}(x,v')\notag\\
&\quad + (\varepsilon_1\varepsilon_2)^{-1} [F^{(2)}(x,u')(F(x,v')- F^{(2)}(x,v'))   + F^{(1)}(x,v')(F(x,u') - F^{(1)}(x,u')) ]\notag\\
& = V^{(1)}(x,u')V^{(2)}(x,v')-V^{(1)}(x,v')V^{(2)}(x,u')\notag\\
&\quad +V^{(1)}(x,v')V^{(2)}(x,u')+V^{(1)}(x,v')V^{(2)}(x,u')\notag\\
&= V^{(1)}(x,u')V^{(2)}(x,v')+V^{(1)}(x,v')V^{(2)}(x,u'),
\end{align}
where we applied Lemma~\ref{lemma:1st appro}. Therefore, we have
$$
   \lim\limits_{\varepsilon\rightarrow 0} \|\mathcal{H}_1\|_{C(\Omega\times\R^n)} =0.
$$
Similarly, replacing $u',v'$ by $u,v$ in \eqref{2nd identity}, we obtain 
\begin{align*}
 &\lim_{\varepsilon\rightarrow 0}(\varepsilon_1\varepsilon_2)^{-1}( F(x,u)F(x,v) - F^{(2)}(x,u) F^{(2)}(x,v) -  F^{(1)}(x,u) F^{(1)}(x,v))\\
 &= V^{(1)}(x,u)V^{(2)}(x,v)+V^{(1)}(x,v)V^{(2)}(x,u),
\end{align*}
and then we get
$$
\lim\limits_{\varepsilon\rightarrow 0} \|\mathcal{H}_1\|_{C(\Omega\times\R^n)} =0.
$$
Thus, combining the above two limits of $\mathcal{H}_j$, it clearly implies that the right-hand side of \eqref{mathcal G} approaches to zero as $\varepsilon$ goes to zero. This completes the proof.
\end{proof}

From Lemma~\ref{lemma:1st appro} and Lemma~\ref{lemma:2nd appro}, now we can denote the solution $V^{(k)}$ for \eqref{Boltzmann equ V} by the first derivative of $F$,  $\p_{\varepsilon_k}F|_{\varepsilon=0}$, that is, 
$$
    V^{(k)}=\p_{\varepsilon_k}F|_{\varepsilon=0}.
$$ 
In particular, the solution $V^{(k)}$ takes the form
$$
V^{(k)}(x,v) = g_k(x-\tau_-(x,v)v,v)  
$$
for any $(x,v)\in \Omega\times \R^3$ for $k=1,2$.

Moreover, we can also denote the solution $W$ for \eqref{Boltzmann equ linear W} by the second derivative $\p_{\varepsilon_1}\p_{\varepsilon_2} F|_{\varepsilon=0}$ , that is, 
\begin{equation} \label{eq:Wdef}
W= \p_{\varepsilon_1}\p_{\varepsilon_2} F|_{\varepsilon=0}.
\end{equation}
In addition, the solution $W$ can be expressed as
\begin{equation} \label{eq:Sint} 
W(x,v)= \int^{\tau_-(x,v)}_0 	 S(x-sv, v)\,ds ,
\end{equation}
where $S$ is 	defined in \eqref{eq:S}.

\subsection{Linearization of the boundary map}
We first extend the Boltzmann solution to the boundary $\Gamma_+$ in Lemma~\ref{lemma trace} and then show the boundedness of the operator $\mathcal{A}$ in Proposition~\ref{prop A}. Finally, we turn to illustrate the linearization of $\mathcal{A}$ in Lemma~\ref{lemma bdry asym}.

In the following lemma, we show a trace theorem in the spirit of \cite{Ces84, Ces85}, see also \cite{IKun,CS98}.
\begin{lemma}\label{lemma trace}
Let $F$ be the solution of \eqref{Boltzmann equ}. Suppose that $B$ satisfies \eqref{constraint B}. For all $(x ,v)\in\Gamma_+$, the limit
$F(x,v) = \lim\limits_{t\downarrow 0} F(x-tv,v)$  
exists and, moreover, $F\in C(\Gamma_\pm)$.
\end{lemma}
\begin{proof}
Since $g\in C(\Gamma_-)$ and \eqref{constraint B}, we have $F$ and $v\cdot\nabla_x F$ are bounded in $\Omega\times\R^n$.
Suppose that $B$ satisfies \eqref{constraint B}. For any $(x ,v)\in\Gamma_+$, by the fundamental theorem of calculus, we conclude
\begin{align}\label{def F+}
    F(x ,v) = F(x-tv,v) + \int^{t}_0 v\cdot\nabla_x F(x-sv,v)ds,
\end{align}
We have
$$
|F(x ,v) - F(x-tv,v)|= |\int^{t}_0 v\cdot\nabla_x F(x-sv,v)ds| \leq t\|v\cdot\nabla_x F\|_{C(\Omega\times\R^n)}.
$$
This implies that $F|_{\Gamma_+}$ in \eqref{def F+} is well-defined and $F(x,v) = \lim\limits_{t\downarrow 0} F(x-tv,v)$ for all $(x ,v)\in\Gamma_+$.
\end{proof}

Lemma~\ref{lemma trace} immediately implies the following result.
\begin{proposition}\label{prop A}
The boundary operator $\mathcal{A}$ is a bounded map $\mathcal{A}: \mathcal{X}\rightarrow C(\Gamma_+)$, where $\mathcal{X}$ is defined in \eqref{small boundary}.
\end{proposition}

Thus, Lemma~\ref{lemma:2nd appro} and Lemma~\ref{lemma trace} lead to the following result right away.
\begin{lemma}\label{lemma bdry asym}
For sufficiently small constants $\varepsilon_1,\,\varepsilon_2>0$ and $g_1,\,g_2\in C(\Gamma_-)$, we have
\begin{align*}
\lim_{\varepsilon \rightarrow 0 }\|(\varepsilon_1 \varepsilon_2)^{-1} (\mathcal{A} (\varepsilon_1 g_1+\varepsilon_2 g_2) - \mathcal{A}(\varepsilon_2 g_2)-\mathcal{A}(\varepsilon_1 g_1))- W\|_{C(\Gamma_+)} =0.
\end{align*}
\end{lemma}
\begin{remark}\label{remark W}
Based on the definition~\eqref{eq:Wdef} and Lemma~\ref{lemma bdry asym}, the outgoing boundary value $W|_{\Gamma_+}$ can be reconstructed as
\begin{equation} \label{eq:WA}
W|_{\Gamma_+} =   \lim_{\varepsilon\rightarrow 0} (\varepsilon_1\varepsilon_2)^{-1}(\mathcal{A}(\varepsilon_1 g_1 + \varepsilon_2 g_2) -\mathcal{A}(\varepsilon_2 g_2)-\mathcal{A}(\varepsilon_1 g_1 )).
\end{equation}
We obtain that if $\mathcal{A}_1(g) = \mathcal{A}_2(g)$ for all $g\in \mathcal{X}$, then the boundary operator $\mathcal{A}$ uniquely determines the function $W|_{\Gamma_+}$.
\end{remark}

\subsection{Proof of Theorem~\ref{main thm}}
We are ready to prove the first uniqueness result in this paper.
\begin{proof}[Proof of Theorem~\ref{main thm}]
When $|\varepsilon|$ is sufficiently small, the boundary value problem
\begin{align*}
\left\{ \begin{array}{ll} 
	v\cdot \nabla_x F_j  = Q_j(F_j,F_j)& \hbox{in } \Omega\times\R^n,\\
	F_j=\varepsilon_1 g_1+\varepsilon_2 g_2  &  \hbox{on }\Gamma_-,\\
\end{array}\right. 
\end{align*}
has a unique small solution $F_j=F_j(x,v;\varepsilon)\in C(\Omega\times\R^n)$.  
For $j=1,2$, differentiating the above equation with respect to $\varepsilon_k$ and taking $\varepsilon=0$, the function $V^{(k)}=\p_{\varepsilon_k}F_j|_{\varepsilon=0}$ is the solution to the problem
\begin{align}\label{proof:Boltzmann equ linear 1}
\left\{ \begin{array}{ll} 
v\cdot \nabla_x V^{(k)} = 0 & \hbox{in } \Omega\times\R^n,\\
V^{(k)}=  g_k& \hbox{on }\Gamma_-.\\
\end{array}\right. 
\end{align}  
In addition, we also have $$W_j= \p_{\varepsilon_1}\p_{\varepsilon_2} F_j|_{\varepsilon=0},$$ satisfying the problem 
\begin{align}\label{proof:Boltzmann equ linear W}
\left\{ \begin{array}{ll} 
v\cdot \nabla_x W_j= S_j(x,v)  & \hbox{in } \Omega\times\R^n,\\
W_j= 0& \hbox{on }\Gamma_-,\\
\end{array}\right. 
\end{align}
where the source term is 
\begin{align*}
S_j(x,v)&=\int_{\R^n}\int_{\mathbb{S}^{n-1}} B_j(v,u,\omega) [V^{(1)}(x,v')V^{(2)}(x,u')+V^{(1)}(x,u')V^{(2)}(x,v') \\
&\hskip4cm - V^{(1)}(x,u)V^{(2)}(x,v)-V^{(1)}(x,v)V^{(2)}(x,u)]\,d\omega du.
\end{align*}
In particular, the solution $W_j$ can be written as
\begin{align*}
W_j(x,v)= \int^{\tau_-(x,v)}_0  S_j(x-sv, v)\,ds .
\end{align*}

Since the maps $\mathcal{A}_1 (g)=\mathcal{A}_2(g)$ for all boundary data $g\in \mathcal{X}$, from Remark~\ref{remark W}, we have $$W_1|_{\Gamma_+}=W_2|_{\Gamma_+}.$$ Thus, for any $(x,v)\in\Gamma_+$, one can derive that
\begin{align}\label{B1B2}
    0&= \int_0^{\tau_-(x,v)}  (S_1-S_2)(x-sv,v)\,ds \notag\\
    &=  \int_0^{\tau_-(x,v)} \LC \int_{\R^n}\int_{\mathbb{S}^{n-1}}(B_1-B_2)(v,u,\omega)P(x-sv,v,u,\omega)d\omega du\RC ds,
\end{align}
where we denote $P$ by
\begin{align*} 
    P(x,v,u,\omega) &:= V^{(1)}(x,v')V^{(2)}(x,u')+V^{(1)}(x,u')V^{(2)}(x,v') \\
      &\hskip2cm- V^{(1)}(x,u)V^{(2)}(x,v)-V^{(1)}(x,v)V^{(2)}(x,u).
\end{align*}

Fixing a nonzero vector $v_0\in \R^n$. Note that the equation \eqref{proof:Boltzmann equ linear 1} is independent of the kernel. Thus, we can freely choose 
$$
    V^{(1)}(x,v)= e^{|v-v_0|^2}\ \ \hbox{and }\ V^{(2)}(x,v)\equiv 1.
$$ 
By substituting them into \eqref{B1B2}, we obtain 
\begin{align}\label{}
    0=  \int_0^{\tau_-(x,v)} \int_{\R^n}\int_{\mathbb{S}^{n-1}} (B_1-B_2)(v,u,\omega) P(v,u,\omega)\,  d\omega duds,
\end{align}
with 
$$
    P(v,u,\omega):=P(x,v,u,\omega) = e^{|u'-v_0|^2}+e^{|v'-v_0|^2}-e^{|v-v_0|^2}-e^{|u-v_0|^2}.
$$

By applying \eqref{velocities after collision} with incoming velocities $u,v_0$, we obtain
$$
    |u'-v_0|^2= |u-v_0|^2 - |(v_0-u)\cdot w|^2\ \ \hbox{and} \ \ |v'-v_0|^2=|(v_0-u)\cdot \omega|^2,
$$
and then we apply these identities to derive that
\begin{align*}
    P(v_0,u,\omega) &= e^{|u-v_0|^2 - |(v_0-u)\cdot \omega|^2}+e^{|(v_0-u)\cdot \omega|^2}-1-e^{|u-v_0|^2}\\
    &=(1-e^{-|(v_0-u)\cdot \omega|^2}) (e^{|(v_0-u)\cdot \omega|^2}-e^{|u-v_0|^2}).
\end{align*} 
We also denote the subspace $N_{v_0u}$ of the unit sphere by
$$
    N_{v_0u}=\left\{z\in\mathbb{S}^{n-1}:\ z\perp(v_0-u)\ \hbox{or }z =\pm {v_0-u\over |v_0-u|}\right\}.
$$ 
Then $P=0$ if and only if $\omega \in N_{v_0u}$.
On the other hand, if $\omega\notin N_{v_0u}$, then $P(v_0,u,\omega)<0$. Thus, we have $P(v_0,u,\omega)\leq 0$.
From the monotonicity condition $B_1\geq B_2$, we further get 
$$
    (B_1-B_2)(v_0,u,\omega) P(v_0,u,\omega)\leq 0 \ \ \hbox{for all }x\in\Omega,\  u\in\R^n, \ \ \omega\in\mathbb{S}^{n-1},
$$ which implies that
\begin{align*}
    0&=  \int_0^{\tau_-(x,v_0)}\int_{\R^n}\int_{\mathbb{S}^{n-1}} (B_1-B_2)(v_0,u,\omega) P(v_0,u,\omega) \, d\omega duds\leq 0. 
\end{align*}
Since $(B_1-B_2)P\leq 0$ and $B_j$ and $P$ are continuous, we have
$$
(B_1-B_2)(v_0,u,\omega) P(v_0,u,\omega)=0 \ \ \hbox{for all }	u\in\R^n, \ \ \omega\in\mathbb{S}^{n-1}.
$$
Therefore, from the fact that $P<0$ for every $\omega\notin N_{v_0u}$ and the continuity of $B_j$, for any nonzero fixed vector $v_0$, we can conclude that $B_1(v_0,\cdot,\cdot)=B_2(v_0,\cdot, \cdot)$ in $\R^n\times\mathbb{S}^{n-1}$. Since $v_0$ is arbitrary, we can obtain $B_1=B_2$, which completes the proof.
 
\end{proof}


\section{A Reconstruction formula}\label{sec:formula}
In this section, we derive a reconstruction formula for the kernel $B$ by making use of special solutions $V^{(k)}$, that concentrate near the incoming directions. Furthermore, by using this established formula, we can show that the uniqueness results hold in two special cases, as stated in Corollary~\ref{Coro}.

Recall that for any $u,v\in\mathbb{R}^n$ and any $\omega\in\mathbb{S}^{n-1}$, we will insist two basic properties of the collision kernel $B$, as stated in \eqref{assumption B}:
\begin{enumerate}
\item $B$ is symmetric in incoming velocities: $B(v,u,\omega) = B(u,v,\omega)$;
\item $B$ is an even function of $\omega$: $B(v,u,-\omega)=B(v,u,\omega)$. 
\end{enumerate}
 
We recall that $W$ solves the boundary value problem~\eqref{Boltzmann equ linear W}, hence can be written as line integrals of the internal source $S$, see~\eqref{eq:Sint}. We are interested in the case where $S$ does not depend on $x$. Then~\eqref{eq:Sint} reduces to 
\begin{equation} \label{eq:WS}
W(x,v) = \tau_-(x,v) S(v)
\end{equation}
for $(x,v)\in\Gamma_+$. Therefore, we can recover $S(v)$ from $\mathcal{A}$ as long as there is at least one $x\in\partial\Omega$ such that $\tau_-(x,v)\neq 0$.

In view of~\eqref{eq:S}, the way to make $S$ independent of $x$, as in the proof of Theorem~\ref{main thm}, is to choose transport solutions $V^{(1)}$ and $V^{(2)}$ in~\eqref{eq:S} that only depend on $v$, that is, 
\begin{align} 
S(v) &:= \int_{\R^n}\int_{\mathbb{S}^{n-1}} B(v,u,w) [V^{(1)}(v')V^{(2)}(u')+V^{(1)}(u')V^{(2)}(v') \nonumber \\
&\hskip4cm - V^{(1)}(v)V^{(2)}(u)-V^{(1)}(u)V^{(2)}(v)]\,dwdu, \label{eq:Sv}
\end{align}
where $V^{(k)} = V^{(k)}(v)$ automatically solves the transport equation $v\cdot\nabla_x V^{(k)}=0$ in $\Omega\times\mathbb{R}^n$, $k=1,2$. 
It suffices to construct special transport solutions to extract information on $B$. The $x$-independent solutions are sufficient for our purpose since $B=B(v,u,\omega)$ does not depend on $x$.

Pick three distinct vectors $u_0, v_0, v_\ast\in\mathbb{R}^n$. 
We formally choose $V^{(1)} = \delta_{v_0}$, $V^{(2)} = \delta_{u_0}$ in~\eqref{eq:Sv}, then multiply~\eqref{eq:Sv} by the delta function $\delta_{v_\ast}(v)$ and integrate in $v$ over $\mathbb{R}^n$ to obtain
$$ 
S(v_\ast, v_0, u_0) := I_1 + I_2 + I_3 + I_4,
$$
where $I_j, j=1,\cdots,4$, is defined by
\begin{align*}
I_1 & :=  \int_{\R^n} \int_{\R^n}\int_{\mathbb{S}^{n-1}} B(v,u,\omega) \delta_{v_0}(v') \delta_{u_0}(u') \delta_{v_\ast}(v) \,d\omega du dv;\\
I_2 & := \int_{\R^n} \int_{\R^n}\int_{\mathbb{S}^{n-1}} B(v,u,\omega) \delta_{v_0}(u') \delta_{u_0}(v') \delta_{v_\ast}(v) \,d\omega du dv; \\
I_3 & := - \int_{\R^n} \int_{\R^n}\int_{\mathbb{S}^{n-1}} B(v,u,\omega) \delta_{v_0}(v) \delta_{u_0}(u) \delta_{v_\ast}(v) \,d\omega du dv; \\
I_4 & := - \int_{\R^n} \int_{\R^n}\int_{\mathbb{S}^{n-1}} B(v,u,\omega) \delta_{v_0}(u) \delta_{u_0}(v) \delta_{v_\ast}(v) \,d\omega du dv.
\end{align*}
The value $S(v_\ast, v_0, u_0)$ can be calculated from $\mathcal{A}$ using~\eqref{eq:WS} and ~\eqref{eq:WA} with $g_1 = V^{(1)}|_{\Gamma_-} = \delta_{v_0}|_{\Gamma_-}$ and $g_2 = V^{(2)}|_{\Gamma_-} = \delta_{u_0}|_{\Gamma_-}$. We will divide the calculation of the four integrals into several lemmas and propositions. We note that the arguments below can be made rigorously by replacing the delta functions by limits of some smooth cut-off functions.

\subsection{Preliminaries}

We remark that $u'$ and $v'$ in the integrands of $I_1$ and $I_2$ should be interpreted as functions of $u$ and $v$, as was defined in~\eqref{velocities after collision}. Since the map $(u,v)\mapsto (u',v')$ in~\eqref{velocities after collision} is an isometry for each $\omega\in\mathbb{S}^{n-1}$, one can invert it to write $(u,v)$ as functions of $(u',v')$ as well. Explicitly, 
\begin{equation} \label{eq:uv}
u = u(u',v',\omega) := u' - [(u'-v')\cdot\omega]\omega, \quad\quad\quad  v = v(u',v',\omega) := v' + [(u'-v')\cdot\omega]\omega.
\end{equation}
Some basic properties of these functions are recorded below.
\begin{lemma} \label{lemma:uvproperty}
The functions $u = u(u',v',\omega)$ and $v = v(u',v',\omega)$ defined in~\eqref{eq:uv} satisfy
\begin{enumerate}
\item $u(u',v',-\omega) = u(u',v',\omega)$ and $v(u',v',-\omega) = v(u',v',\omega)$;
\item $u(v',u',\omega) = v(u',v',\omega)$ and $v(v',u',\omega) = u(u',v',\omega)$. 
\end{enumerate}
\end{lemma}
\begin{proof} These are straightforward calculations:
\begin{enumerate}
\item $u(u',v',-\omega) = u' - [(u'-v')\cdot(-\omega)](-\omega) = u' - [(u'-v')\cdot\omega]\omega = u(u',v',\omega)$.\\
$v(u',v',-\omega) = v' + [(u'-v')\cdot(-\omega)](-\omega) = v' + [(u'-v')\cdot\omega]\omega = v(u',v',\omega)$.
\item $u(v',u',\omega) = v' - [(v'-u')\cdot\omega]\omega = v' + [(u'-v')\cdot\omega]\omega = v(u',v',\omega)$.\\
$v(v',u',\omega) = u' + [(v'-u')\cdot\omega]\omega = u' - [(u'-v')\cdot\omega]\omega = u(u',v',\omega)$.
\end{enumerate}
\end{proof}

We study the solvability of two equations for $\omega$, which will be used later to compute $I_1$ and $I_2$. For any nonzero vector $u\in\mathbb{R}^n$, we denote by $\hat{u}$ the unit vector along the direction of $u$, that is, $\hat{u} := \frac{u}{|u|} \in\mathbb{S}^{n-1}$.

\begin{lemma} \label{lemma:solvability1}
Let $u_0, v_0, v_\ast\in\mathbb{R}^n$ be three distinct vectors. 
\begin{enumerate}
	\item The equation $v_\ast = v(u_0,v_0,\omega)$ admits solutions $\omega\in\mathbb{S}^{n-1}$ if and only if 
	\begin{equation} \label{eq:rel1}
	(v_\ast - v_0)\cdot(u_0 - v_0) = |v_\ast - v_0|^2.
	\end{equation}
	When~\eqref{eq:rel1} holds, the solutions are $\omega=\pm\omega_1$, where $\omega_1 := \widehat{(v_\ast - v_0)}$.
		
	\item  The equation $v_\ast = v(v_0,u_0,\omega)$ admits solutions $\omega\in\mathbb{S}^{n-1}$ if and only if 
	\begin{equation} \label{eq:rel2}
	-(v_\ast - u_0)\cdot(u_0 - v_0) = |v_\ast - u_0|^2.
	\end{equation}
	When~\eqref{eq:rel2} holds, the solutions are $\omega=\pm\omega_2$, where $\omega_2 := \widehat{(v_\ast - u_0)}$.
\end{enumerate} 
\end{lemma}
\begin{proof}
First, in view of~\eqref{eq:uv}, the equation $v_\ast = v(u_0,v_0,\omega)$ is equivalent to
$$
v_\ast - v_0 = [(u_0-v_0)\cdot\omega]\omega.
$$
If a solution $\omega\in\mathbb{S}^{n-1}$ exists, matching the directions implies $\omega = \pm \omega_1$, and matching the amplitudes implies 
$$
|v_\ast - v_0| = |(u_0-v_0)\cdot\omega_1| = |(u_0-v_0)\cdot\frac{v_\ast - v_0}{|v_\ast - v_0|}|
$$ 
which is the desired relation~\eqref{eq:rel1}. Conversely, if~\eqref{eq:rel1} holds, one has
$$
[(u_0-v_0)\cdot (\pm\omega_1)] (\pm\omega_1) = [(u_0-v_0)\cdot \frac{v_\ast - v_0}{|v_\ast - v_0|}] \frac{v_\ast - v_0}{|v_\ast - v_0|} = v_\ast - v_0,
$$
indicating that $\omega=\pm\omega_1$ are solutions.

Second, switching the roles of $u_0$ and $v_0$ in \eqref{eq:rel1} yields \eqref{eq:rel2}.
\end{proof}

Next, we prove that the relations~\eqref{eq:rel1} and~\eqref{eq:rel2} are actually equivalent, and $\omega_1$ is orthogonal to $\omega_2$ whenever they exist. This is the content of the next lemma. 
\begin{lemma} \label{lemma:ortho}
Let $u_0, v_0, v_\ast\in\mathbb{R}^n$ be distinct vectors. Set $\omega_1 := \widehat{(v_\ast - v_0)}$ and $\omega_2 := \widehat{(v_\ast - u_0)}$. Then both~\eqref{eq:rel1} and~\eqref{eq:rel2} are equivalent to the orthogonality relation
\begin{equation} \label{eq:rel3}
(v_\ast - v_0)\cdot (v_\ast - u_0) = 0.
\end{equation}
In particular, if one of $\omega_1$ and $\omega_2$ exists, so does the other, and we have $\omega_1\cdot\omega_2 = 0$.
\end{lemma}

\begin{proof}
The relation~\eqref{eq:rel1} is equivalent to
\begin{align*}
0 & = (v_\ast - v_0)\cdot(u_0 - v_0) - |v_\ast - v_0|^2 = (v_\ast - v_0)\cdot(u_0 - v_0) - (v_\ast - v_0)\cdot(v_\ast - v_0) \\
 & = (v_\ast - v_0)\cdot(u_0 - v_\ast),
\end{align*}
which is~\eqref{eq:rel3}. Switching the roles of $u_0$ and $v_0$ gives the equivalence of~\eqref{eq:rel2} and~\eqref{eq:rel3}.

The solution $\omega_1$ exists if and only if~\eqref{eq:rel1} holds, which is equivalent to~\eqref{eq:rel3} thus~\eqref{eq:rel2}. The latter holds if and only if $\omega_2$ exists. Finally, the directions of $\omega_1$ and $\omega_2$ are identical to those of $v_\ast - v_0$ and $v_\ast - u_0$, respectively, hence $\omega_1 \cdot \omega_2 = 0$ whenever they exist.
\end{proof}

\begin{lemma} \label{lemma:uvomega}
Let $u_0, v_0, v_\ast\in\mathbb{R}^n$ be distinct vectors and $\omega_1$, $\omega_2$ be defined as above. Suppose~\eqref{eq:rel3} holds so that $\omega_1$ and $\omega_2$ exist. We have
\begin{enumerate}
\item $u(u_0,v_0,\omega_2) = v(u_0,v_0,\omega_1) = v_\ast$;
\item $v(u_0,v_0,\omega_2) = u(u_0,v_0,\omega_1) = u_0 + v_0 - v_\ast$.
\end{enumerate} 
\end{lemma}

\begin{proof}
\begin{enumerate}
\item We compute
\begin{align*}
u(u_0,v_0,\omega_2) & = u_0 - [(u_0 - v_0)\cdot\omega_2]\omega_2 = u_0 - [(u_0 - v_0)\cdot \frac{v_\ast - u_0}{|v_\ast - u_0|}]\frac{v_\ast - u_0}{|v_\ast - u_0|} \\
 & = u_0 + v_\ast - u_0 = v_\ast,
\end{align*}
where the third equality comes from~\eqref{eq:rel2}. On the other hand,
\begin{align*}
v(u_0,v_0,\omega_1) & = v_0 + [(u_0 - v_0)\cdot\omega_1]\omega_1 = v_0 + [(u_0 - v_0)\cdot \frac{v_\ast - v_0}{|v_\ast - v_0|}]\frac{v_\ast - v_0}{|v_\ast - v_0|} \\
 & = v_0 + v_\ast - v_0 = v_\ast,
\end{align*}
where the third equality comes from~\eqref{eq:rel1}.

\item Likewise, we compute
\begin{align*}
v(u_0,v_0,\omega_2) & = v_0 + [(u_0 - v_0)\cdot\omega_2]\omega_2 = v_0 + [(u_0 - v_0)\cdot \frac{v_\ast - u_0}{|v_\ast - u_0|}]\frac{v_\ast - u_0}{|v_\ast - u_0|} \\
 & = v_0 - (v_\ast - u_0) = v_0 - v_\ast + u_0,
\end{align*}
On the other hand,
\begin{align*}
u(u_0,v_0,\omega_1) & = u_0 - [(u_0 - v_0)\cdot\omega_1]\omega_1 = u_0 - [(u_0 - v_0)\cdot \frac{v_\ast - v_0}{|v_\ast - v_0|}]\frac{v_\ast - v_0}{|v_\ast - v_0|} \\
 & = u_0 - (v_\ast - v_0) = u_0 - v_\ast + v_0.
\end{align*}
\end{enumerate}
\end{proof}

\subsection{Calculation of $I_1$--$I_4$.}
We are ready to compute the integrals $I_k$, $k=1,\cdots, 4$.  
\begin{proposition} \label{prop:I1}
Let $u_0, v_0, v_\ast\in\mathbb{R}^n$ be three distinct vectors, then
\begin{enumerate}
\item  
$$
I_1 =
\left\{
\begin{array}{ll}
|(u_0-v_0)\cdot\omega_1|^{-n}B(v_\ast, u_0 + v_0 - v_\ast,\omega_1) & \quad\quad\quad \text{ if } \eqref{eq:rel3} \text{ holds,} \\
0 & \quad\quad\quad \text{ otherwise};
\end{array}
\right.
$$
\item
$$
I_2 =
\left\{
\begin{array}{ll}
|(u_0-v_0)\cdot\omega_2|^{-n}B(v_\ast, u_0 + v_0 - v_\ast,\omega_2) & \quad\quad\quad \text{ if } \eqref{eq:rel3} \text{ holds,} \\
0 & \quad\quad\quad \text{ otherwise}.
\end{array}
\right.
$$

\end{enumerate}
\end{proposition}

\begin{proof}
We make the change of variable $(u,v)\mapsto (u',v')$ in $I_1$. The resulting Jacobian is $1$ since the transformation is isometric for each $\omega\in\mathbb{S}^{n-1}$. Therefore,   
\begin{align*}
I_1 & =  \int_{\R^n} \int_{\R^n} \int_{\mathbb{S}^{n-1}} B(v,u,\omega) \delta_{v_0}(v') \delta_{u_0}(u') \delta_{v_\ast}(v) \,d\omega du dv \\
 & =  \int_{\R^n} \int_{\R^n} \int_{\mathbb{S}^{n-1}} B(v(u',v',\omega),u(u',v',\omega),\omega) \delta_{v_0}(v') \delta_{u_0}(u') \delta_{v_\ast}(v(u',v',\omega)) \,d\omega du' dv' \\
 & =  \int_{\mathbb{S}^{n-1}} B(v(u_0,v_0,\omega),u(u_0,v_0,\omega),\omega) \delta_{v_\ast}(v(u_0,v_0,\omega)) \,d\omega.
\end{align*}
Thus, if $v_\ast \neq v(u_0,v_0,\omega)$, then $I_1=0$.

We have seen that the equation $v_\ast = v(u_0,v_0,\omega)$ have solutions $\omega=\pm\omega_1$ if and only if~\eqref{eq:rel3} holds. Therefore, combining with the change of variable, we can derive
\begin{align*}
I_1 & = 2^{-1}|(u_0-v_0)\cdot\omega_1|^{-n}(B(v(u_0,v_0,\omega_1),u(u_0,v_0,\omega_1),\omega_1) + B(v(u_0,v_0,-\omega_1),u(u_0,v_0,-\omega_1),-\omega_1)) \\
 & = 2^{-1}|(u_0-v_0)\cdot\omega_1|^{-n}( B(v(u_0,v_0,\omega_1),u(u_0,v_0,\omega_1),\omega_1) + B(v(u_0,v_0,\omega_1),u(u_0,v_0,\omega_1),-\omega_1) )\\
 & =  |(u_0-v_0)\cdot\omega_1|^{-n}B(v(u_0,v_0,\omega_1),u(u_0,v_0,\omega_1),\omega_1) \\
 & =|(u_0-v_0)\cdot\omega_1|^{-n}B(v_\ast,u_0 + v_0 - v_\ast,\omega_1),
\end{align*}
where the term $2^{-1}|(u_0-v_0)\cdot\omega_1|^{-n}$ comes from the Jacobian, the second equality follows from Lemma~\ref{lemma:uvproperty}, and the third equality is valid since $B(v,u,\omega)$ is assumed to be an even function of $\omega$, and the last equality follows from Lemma~\ref{lemma:uvomega}.

To obtain result (2), one just switches $u_0$ with $v_0$ in the above argument and applies the properties of $B$.  
\end{proof}

\begin{proposition} \label{prop:I3I4}
Let $u_0, v_0, v_\ast\in\mathbb{R}^n$ be three distinct vectors. Then $I_3 = I_4 = 0$.
\end{proposition}
\begin{proof}
$I_3$ can be computed as
\begin{align*}
I_3 & = - \int_{\R^n} \int_{\R^n}\int_{\mathbb{S}^{n-1}} B(v,u,\omega) \delta_{v_0}(v) \delta_{u_0}(u) \delta_{v_\ast}(v) \,d\omega du dv  \\
 & = - \int_{\mathbb{S}^{n-1}} B(v_0,u_0,\omega) \delta_{v_\ast}(v_0) \,d\omega  \\
 & = - \left( \int_{\mathbb{S}^{n-1}} B(v_0,u_0,\omega) \,d\omega \right) \delta_{v_\ast}(v_0) = 0,
\end{align*}
where the last identity is valid since $v_\ast, v_0$ are distinct vectors.

Similarly, we have
\begin{align*}
I_4 & = - \int_{\R^n} \int_{\R^n}\int_{\mathbb{S}^{n-1}} B(v,u,\omega) \delta_{v_0}(u) \delta_{u_0}(v) \delta_{v_\ast}(v) \,d\omega du dv  \\
 & = - \int_{\mathbb{S}^{n-1}} B(u_0,v_0,\omega) \delta_{v_\ast}(u_0) \,d\omega  \\
 & = - \left( \int_{\mathbb{S}^{n-1}} B(v_0,u_0,\omega) \,d\omega \right) \delta_{v_\ast}(u_0) = 0,
\end{align*}
where the last identity is also valid since $v_\ast, u_0$ are distinct vectors.
\end{proof}

\subsection{Recovery of $B$}
We are ready to derive the reconstruction formula of the kernel $B$ in Theorem~\ref{thm:ReconstructionFormula} when $B$ satisfies both conditions \eqref{constraint B}-\eqref{assumption B}. In addition, with the established formula, we will discuss the uniqueness result under two different constraints as well as \eqref{constraint B}-\eqref{assumption B}, while Theorem~\ref{main thm} only requires \eqref{constraint B}.

\begin{proof}[Proof of Theorem~\ref{thm:ReconstructionFormula}]
Recall that 
$$ 
S(v_\ast, v_0, u_0) = I_1 + I_2 + I_3 + I_4,
$$
where $I_1$--$I_4$ have been computed in Proposition~\ref{prop:I1} and Proposition~\ref{prop:I3I4}.

Since $u_0,v_0,v_\ast$ are distinct, $I_3$ and $I_4$ vanish. We have, from Proposition~\ref{prop:I1}, that
\begin{align*}
S(v_\ast, v_0, u_0)  &= I_1 + I_2  \\
&= 
\left\{
\begin{array}{ll}
\sum_{k=1}^2|(u_0-v_0)\cdot\omega_k|^{-n} B(v_\ast, u_0 + v_0 - v_\ast,\omega_k)  & \quad\quad \text{ if } \eqref{eq:rel3} \text{ holds,} \\
0 & \quad\quad \text{ otherwise}.
\end{array}
\right.
\end{align*}
Given any $(a,b,\theta)\in\mathbb{R}^n\times\mathbb{R}^n\times\mathbb{S}^{n-1}$, we can choose
$$
v_\ast = a, \quad v_0 = a - [(a-b)\cdot\theta]\theta, \quad u_0 = b + [(a-b)\cdot\theta]\theta.
$$
Then
$$
    u_0 + v_0 - v_\ast = b,\quad \omega_1=\pm\theta.
$$

One can check~\eqref{eq:rel3} holds for such triple $(v_\ast,v_0,u_0)$. Moreover, $v_\ast - u_0 = (a-b)-[(a-b)\cdot\theta]\theta$ hence
$$
\omega_2 = \widehat{(v_\ast - u_0)} = \widehat{P_{\theta^\perp}(a-b)}, 
$$
where $$P_{\theta^\perp}(a-b) := (a-b)-[(a-b)\cdot\theta]\theta.$$
In addition, we also have
$$
     |(u_0-v_0)\cdot\omega_1|^2 = |(a-b)\cdot\theta|^2,\quad  |(u_0-v_0)\cdot\omega_2|^2=|a-b|^2-|(a-b)\cdot\theta|^2.
$$
Therefore, we get
\begin{align} \label{eq:SInTemrsOfB}
&S(a,a - [(a-b)\cdot\theta]\theta, b + [(a-b)\cdot\theta]\theta) \notag\\
&= |(u_0-v_0)\cdot\omega_1|^{-2} B(a, b, \theta) +  |(u_0-v_0)\cdot\omega_2|^{-2}B(a, b, \widehat{P_{\theta^\perp}(a-b)})\notag\\
&= |(a-b)\cdot\theta|^{-2} (B(a, b, \theta)+ (|a-b|^2-|(a-b)\cdot\theta|^2)^{-1} B(a, b, \widehat{P_{\theta^\perp}(a-b)}),
\end{align}
which means we can recover the sum on the right hand side. The condition that $v_\ast,u_0,v_0$ being distinct translates to $(a,b,\theta)$ as 
\begin{align}\label{def D}
\mathcal{D}=\{(a,b,\theta)\in\mathbb{R}^n\times\mathbb{R}^n\times\mathbb{S}^{n-1}:\, &(a-b)\cdot\theta\neq 0 \text{ and } a-b \neq [(a-b)\cdot\theta]\theta \notag\\
&\hskip3cm \text{ and } a-b \neq 2[(a-b)\cdot\theta]\theta\}.
\end{align}
As the complement of this set in $\mathbb{R}^n\times\mathbb{R}^n\times\mathbb{S}^{n-1}$ has Lebesgue measure zero, the formula~\eqref{eq:SInTemrsOfB} reconstructs the sum almost everywhere, and eventually everywhere if $B$ is continuous. 
\end{proof}

This reconstruction formula immediately leads to the unique determination of the kernel. Let's recall the statement of Corollary~\ref{Coro}: Suppose that two collision kernels $B_1$ and $B_2$ have identical measurement. Then $B_1=B_2$ in the following two cases: (1) the collision kernel $B=B(v,u)$ is independent of $\omega$; (2)  the monotonicity condition is valid, such as $B_1\geq B_2$.

\begin{proof}[Proof of Corollary~\ref{Coro}]
(1) Since $B$ is independent of $\omega$, we get $$B(a,b) = B(a, b, \theta) = B(a, b, \widehat{P_{\theta^\perp}(a-b)})$$ and thus~\eqref{eq:SInTemrsOfB} uniquely recovers $B$.
	
(2) Since $B_1$ and $B_2$ have the same boundary measurement, from \eqref{eq:SInTemrsOfB}, we have
$$
	0 = |(u_0-v_0)\cdot\omega_1|^{-n}(B_1 - B_2)(a, b, \theta) +  |(u_0-v_0)\cdot\omega_2|^{-2} (B_1 - B_2)(a, b, \widehat{P_{\theta^\perp}(a-b)}).
$$ 
This forces each term on the right hand side to vanish if $(u_0-v_0)\cdot\omega_k\neq 0,\, k=1,2$ since they are non-negative. 
\end{proof}

\vskip1cm
\textbf{Acknowledgment.}
The project was initiated during the IMA workshop ``Mathematics in Optical Imaging" in Spring 2019. The authors would like to thank the hospitality of the IMA.
R.-Y. Lai is partially supported by the NSF grant DMS-1714490. G. Uhlmann is supported in part by the NSF, the Walker Family Endowed Professorship at the University of Washington, and a Si-Yuan Professorship at IAS, HKUST.
Y. Yang is partially supported by the NSF grant DMS-1715178, AMS-Simons travel grant, and the start-up fund from Michigan State University.

\bibliographystyle{abbrv}
\bibliography{boltzmann}

\end{document}